\documentclass[10pt]{amsart}
\usepackage{latexsym, amsmath, amsfonts, amssymb, mathrsfs, fancyhdr, tikz, tikz-cd,accents}
\usepackage{verbatim}
\usepackage{multicol}
\usetikzlibrary{matrix,arrows,decorations.pathmorphing,calc}

\makeatletter
\@namedef{subjclassname@2020}{%
  \textup{2020} Mathematics Subject Classification}
\makeatother

\def\ds{\displaystyle}

\def\e{\varepsilon}

\def\d{\delta}

\def\ch2{\mathbb{C} \mathbb{H}^2}
\def\h2{\mathbb{H}^2}

\def\oshr/2{\cosh \left( \frac{r}{2} \right)}
\def\inhr/2{\sinh \left( \frac{r}{2} \right)}
\def\osh2r/2{\cosh^2 \left( \frac{r}{2} \right)}
\def\inh2r/2{\sinh^2 \left( \frac{r}{2} \right)}
\def\-1/4{- \frac{1}{4}}
\def\H{\mathbb{H}}
\def\C{\mathbb{C}}

\def\P{\mathbb{P}}

\newtheorem{theorem}{Theorem}[section]

\newtheorem{cor}[theorem]{Corollary}

\newtheorem*{question}{Question}

\theoremstyle{definition}

\theoremstyle{remark}
\newtheorem{remark}[theorem]{Remark}

\numberwithin{equation}{section}

\DeclareRobustCommand{\looooongrightarrow}{%
\DOTSB\relbar\joinrel\relbar\joinrel\relbar\joinrel\relbar\joinrel\relbar\joinrel\relbar\joinrel\relbar\joinrel\relbar\joinrel\relbar\joinrel\relbar\joinrel\relbar\joinrel\rightarrow}

\begin{document}

\title[Ratios of Chern numbers for $\C \H$ branched covers]{On ratios of Chern numbers for complex hyperbolic branched covers}

\author{Barry Minemyer}
\address{Department of Mathematics, Computer Science, and Digital Forensics, Commonwealth University - Bloomsburg, Bloomsburg, Pennsylvania 17815}
\email{bminemyer@commonwealthu.edu}


\subjclass[2020]{Primary 55R25, 57R20; Secondary 53C20, 53C24}

\date{\today.}



\begin{abstract}
In this paper we prove that, at least in even complex dimensions, the ratio of Chern numbers for a closed complex hyperbolic branched cover manifold are not all equal to the corresponding ratio of Chern numbers for a closed complex hyperbolic manifold.
This leads to an answer for a question posed by Deraux and Seshadri, and proves that an almost $1/4$-pinched metric constructed by the author in a previous article is not K\"{a}hler.  
\end{abstract}

\maketitle



\section{Introduction}\label{Section:Introduction}

We say that a manifold $M$ is {\it complex hyperbolic} if its universal cover is isometric to complex hyperbolic space $\C \H^n$.
In the literature such manifolds are frequently referred to as {\it complex ball quotients}.
Note that, in this case, $M = \Gamma \setminus \C \H^n$ for some torsion free lattice $\Gamma < \text{PU}(n,1)$.  
In this article we will say that a pair $(M, N)$ is {\it modeled on} $(\C \H^n , \C \H^{n-1})$ if $M$ is complex hyperbolic with complex dimension $n$, $N \subset M$ is embedded and totally geodesic, $N$ has real codimension $2$, and the lift of $N$ to the universal cover of $M$ is isometric to disjoint copies of $\C \H^{n-1}$.  

Recently in \cite{ST} Stover and Toledo proved the following.  
Choose an integer $d \geq 2$, and suppose that $\Gamma$ as above is a cocompact congruence arithmetic lattice of simple type.  
Then there exists a finite cover $(M', N')$ of $(M, N)$ such that the fundamental class $[N'] \in H_{2n-2}(M', \mathbb{Z})$ is $d$-divisible.
Consequently, the $d$-fold cyclic branched cover $X$ of $M'$ about $N'$ exists and is a smooth manifold.  
These complex hyperbolic branched cover manifolds admit a negatively curved K\"{a}hler metric by a result of Zheng \cite{Zheng}, and their existence has led to the solution of several open problems in complex geometry.  
For example, Guenancia and Hamenst\"{a}dt have shown that a large class of these manifolds admit a negatively curved K\"{a}hler-Einstein metric \cite{GH} (see also the related \cite{LafontMinemyer}).
Additionally, the author has shown that these manifolds admit an almost negatively $1/4$-pinched Riemannian metric \cite{MinemyerKahler} and that this same metric has nonpositive curvature operator \cite{MinemyerCurvatureOperators}.

The purpose of this short note is to study the ratio of Chern numbers for these complex hyperbolic branched cover manifolds $X$.  
Our main result in this direction is the following.

\begin{theorem}\label{thm:different chern numbers}
Assume the integer $n \geq 2$ is even, and let $(M, N)$ be modeled on $(\C \H^n, \C \H^{n-1})$ with $M$ closed.  
Then for all but possibly finitely-many integers $d \geq 2$ we have the following.  
Let $(M', N')$ be an arbitrary finite cover of $(M, N)$ such that $[N']$ is $d$-divisible, and let $X$ denote the $d$-fold cyclic branched cover.  
Then there exists a ratio of Chern numbers of $X$ which is not equal to the corresponding ratio of Chern numbers of a complex hyperbolic manifold.
\end{theorem}

For $n=2$ we give a similar, but slightly different argument to prove the following stronger statement.

\begin{theorem}\label{thm:dimension 2}
For $n=2$ and for any integer $d \geq 2$, the $d$-fold cyclic branched cover $X$ described in Theorem \ref{thm:different chern numbers} satisfies $\ds{ c_1^2(X) - 3 c_2(X) \neq 0. }$
\end{theorem}

\begin{remark}
In the proof of Theorem \ref{thm:dimension 2}, we actually prove that
\begin{equation*}
c_1^2(X) - 3c_2(X) = m \frac{(d-1)^2}{2d} \chi(N)
\end{equation*}
where $m$ is the degree of the cover $(M', N') \longrightarrow (M, N)$ and $\chi(N)$ denotes the Euler characteristic of $N$.
\end{remark}

\begin{remark}
It is likely that Theorem \ref{thm:different chern numbers} holds for all dimensions $n$, all choices of ramification degree $d \geq 2$, and all ratios of Chern numbers.  
But the argument that we give below using Hirzebruch's proportionality and signature theorems is only sufficient for even complex dimensions and for some ratio of Chern numbers.
The condition of ``all but finitely many values of $d$" is due to the difficulty in giving exact calculations for the signature of intersections of appropriate submanifolds of $M$ (see Remark \ref{rmk:top Chern class} and the proof of Theorem \ref{thm:different chern numbers}).
A calculation by Deraux \cite{Deraux} when $n=3$ lends further evidence that this holds in all dimensions.  
Also, specific (and much more detailed) calculations for $n=2$ were computed for the original branched-cover examples by Mostow-Siu \cite{MostowSiu} and Hirzebruch \cite{Hirzebruch-LinesAndSurfaces}.
\end{remark}

Theorem \ref{thm:different chern numbers} shows that such branched cover manifolds $X$ are not complex hyperbolic manifolds, but Stover and Toledo gave an alternate argument for this fact for all such branched covers in \cite{ST}.  
The reason why we are interested in the ratio of these Chern numbers is due to a question posed by Deraux and Seshadri in \cite{DerauxSeshadri}.  
To state this question we need to first describe the main result from \cite{DerauxSeshadri}, which is as follows.  
Let $n \geq 1$ be an integer and let $\e > 0$.  
Then there exists $\d ( \e, n) > 0$ such that, if $Y$ is a compact K\"{a}hler $n$-manifold whose sectional curvatures (with respect to the K\"{a}hler metric) are contained in the $\d$-neighborhood of $[-1, -1/4]$, then the ratio of any two Chern numbers of $Y$ must be contained within the $\e$-neighborhood of the corresponding ratio of Chern numbers of a complex hyperbolic manifold.
The question asked by Deraux and Seshadri is the following.

\begin{question}[Deraux and Seshadri \cite{DerauxSeshadri}]
Does the result above hold if one only assumes the existence of an almost $1/4$-pinched Riemannian metric on a compact K\"{a}hler manifold?
\end{question}

This question is interesting due to a result of Hernandez \cite{Hernandez} and Yau and Zheng \cite{YauZheng}.
This result is that, if a compact K\"{a}hler manifold $Y$ admits a negatively $1/4$-pinched Riemannian metric $g$ (which is not assumed to be K\"{a}hler), then $Y$ is isometric to a complex hyperbolic manifold.  
In particular, not only must $g$ be K\"{a}hler, but it must actually equal the standard complex hyperbolic metric.
This result lends some anecdotal evidence that the question of Deraux and Seshadri may possibly be answered in the affirmitive.

The author proved in \cite{MinemyerKahler} that the complex hyperbolic branched cover manifolds of Stover and Toledo admit an almost negatively $1/4$-pinched metric, provided the normal injectivity radius of the branching locus is sufficiently large.  
This result, combined with \cite{DerauxSeshadri} and Theorem \ref{thm:different chern numbers}, give the following two corollaries.

\begin{cor}\label{cor:corollary 1}
The answer to the question of Deraux and Seshadri is ``No".  
More specifically, there exists a sequence of compact K\"{a}hler $n$-manifolds $(X_k)$ whose ratio of Chern numbers are a bounded distance away from those of a complex hyperbolic manifold, but which admit a Riemannian metric with all sectional curvatures contained in the interval $[-1 - (1/k) , (-1/4)+ (1/k)]$. 
\end{cor}

\begin{cor}\label{cor:corollary 2}
The almost negatively $1/4$-pinched Riemannian metric constructed in \cite{MinemyerKahler} (see also \cite{MinemyerCurvatureOperators}) is not K\"{a}hler.
\end{cor}

Note that the two corollaries do not immediately follow from Theorem \ref{thm:different chern numbers}.  
Suppose that $(M, N)$ is modeled on $(\C \H^n, \C \H^{n-1})$ with $M$ closed.  
Let $(M', N')$ be a finite cover so that $[N']$ is $d$-divisible, and let $X$ be the $d$-fold cyclic branched cover of $M'$ about $N'$.
Given $\e > 0$, a metric is constructed in \cite{MinemyerKahler} on such a branched cover $X$ which is $\e$-close to being negatively $1/4$-pinched.  
But, as $\e \to 0$, one needs the normal injectivity radius about $N'$ in $M'$ to approach infinity.  
This requires one to take increasingly large finite covers of the pair $(M, N)$.
In theory, it is possible that the ratio of Chern numbers of $X$ approach those of a complex hyperbolic manifold as one chooses larger covers $(M', N')$ to satisfy this condition on the normal injectivity radius.  
We will see in Section \ref{sect:proofs} that this is not the case and, moreover, the exact opposite happens.

The proofs of Theorems \ref{thm:different chern numbers} and \ref{thm:dimension 2} are simple calculations that use much deeper results, mostly due to Hirzebruch.  
Namely, we use Hirzebruch proportionality \cite{Hirzebruch-Proportionality}, the Hirzebruch signature theorem \cite[Theorem 8.2.2]{Hirzebruch-TopAlgGeomBook} (see also \cite[Theorem 19.4]{MilnorStasheff}), and Hirzebruch's work on the signature of cyclic branched coverings \cite{Hirzebruch-RamifiedCoverings}.
To some experts the calculations in our proofs are trivial and, to that end, the latter two references of Hirzebruch were easy for the author to find.  
But the author, being much more of a geometer than a topologist, was unaware of Hirzebruch proportionality and its immediate consequences.  
References for this were much more difficult for the author to find and, for this reason, in Section \ref{sect:Hirzebruch} below we give a brief overview of Hirzebruch proportionality for Chern numbers and its relationship to the signature and Euler characteristic of a closed complex hyperbolic manifold.  
We then discuss Hirzebruch's work on the signature of a cyclic branched cover in Section \ref{sect:proofs}, where we also prove the theorems and corollaries listed above.

\subsection*{Acknowledgments} This paper is the direct result of a terrific meeting that I had with Domingo Toledo.  
Despite the fact that we had never met, Domingo and his wife Paula invited me to their house for an afternoon while I was visiting Rachel Skipper at the University of Utah.  
This project came from that discussion, and I am incredibly grateful that Domingo took the time to meet with me during my trip.  
I would also like to thank Matthew Stover for helpful emails and for pointing out reference \cite{Hirzebruch-Proportionality}.
Lastly, I would like to thank Jean Lafont for several (really, many) helpful discussions.

\section{Hirzebruch Proportionality for Chern Numbers}\label{sect:Hirzebruch}
The author is unaware of any reference where the content of this section is explicitly recorded in English, although Lafont and Roy discuss a very similar result for Pontrjagin numbers \cite[Theorem A]{LafontRoy}.  
All of the content in this Section is due to Hirzebruch \cite{Hirzebruch-Proportionality}, and is also referenced in the proof of Theorem 22.2.1 in \cite{Hirzebruch-TopAlgGeomBook}.

Let us first fix some notation.  
Following the standard definition of \cite[Chapter 16]{MilnorStasheff}, a {\it partition} of a positive integer $n$ is an unordered sequence $I = i_1, \hdots , i_r$ of positive integers that sum to $n$.
For a manifold $M$ of complex dimension $n$, we use the following notation.
\begin{itemize}
\item $\Sigma(M)$ will denote the signature of $M$.
\item $p_I(M)$ will denote the Pontrjagin number with partition $I$.
\item $c_I(M)$ will denote the Chern number with partition $I$.
\item $\chi(M)$ will denote the Euler characteristic of $M$.
\end{itemize}
Let $\C \P^n$ denote complex projective space of dimension $n$.
The standard background reference is \cite{MilnorStasheff}.

In terms of Chern numbers of closed complex hyperbolic manifolds, Hirzebruch's proportionality can be stated as follows.  

\begin{theorem}[Hirzebruch \cite{Hirzebruch-Proportionality}]\label{thm:Hirzebruch}
Let $M^n$ be a closed complex hyperbolic manifold.  
Then there exists a constant $s$ depending on $M$, but not on the partition $I$, such that
\begin{equation*}
	c_I(M) = s \cdot c_I(\C \P^n).
\end{equation*}
\end{theorem}

\begin{cor}
The ratio of Chern numbers for any closed complex hyperbolic manifold is equal to the corresponding ratio of Chern numbers for $\C \P^n$.  
That is, for any closed complex hyperbolic manifold $M$, 
\begin{equation*}
	\frac{c_I(M)}{c_J(M)} = \frac{c_I(\C \P^n)}{c_J(\C \P^n)}
\end{equation*}	
for all partitions $I, J$ of $n$.
\end{cor}

Theorem \ref{thm:Hirzebruch} and the following corollary are well-known to experts.  
Indeed, references \cite{MostowSiu}, \cite{Hirzebruch-LinesAndSurfaces}, \cite{Zheng}, and \cite{Deraux} all, in some capacity, reference the fact that the ratio of Chern numbers of closed complex hyperbolic manifolds are equal to a constant independent of $M$.
Reference \cite{LafontRoy}, as well as an email from Matthew Stover, confirmed to the author that one uses the dual $\C \P^n$ to compute these ratios of Chern numbers. 
Formulas to compute Chern numbers for $\C \P^n$ are well-known (see, for example, \cite[Example 15.6]{MilnorStasheff}). 

One immediate consequence of Theorem \ref{thm:Hirzebruch} is that, given a closed complex hyperbolic manifold $M$, every Chern number $c_I(M)$ can be written as a known constant times the Euler characteristic $\chi(M)$.
This is done by considering the ratio of $c_I(M)$ with $c_n(M) = \chi(M)$.
For example, if we let $I = n-1, 1$, we have
\begin{equation*}
\frac{c_I(M)}{c_n(M)} = \frac{ {n+1 \choose n-1} {n+1 \choose 1}}{{n+1 \choose n}} = \frac{n(n+1)}{2} \quad \Longrightarrow \quad c_I(M) = \frac{n(n+1)}{2} \chi(M).
\end{equation*}

The following Corollary is easy to find in the literature when $n=2$ (see, for example, \cite{Ville}).  
For $n>2$ the author only knows of the original source \cite{Hirzebruch-Proportionality}.
We give a quick proof of the result since it is crucial for Section \ref{sect:proofs}.

\begin{cor}[Satz 1 (4) of \cite{Hirzebruch-Proportionality}]\label{cor:signature Euler characteristic}
Let $n > 1$ be an even integer, and let $M$ be any closed complex $n$-manifold whose ratio of Chern numbers is equal to those of $\C \P^n$ (for every possible ratio).  
Then the signature $\Sigma(M)$ and Euler characteristic $\chi(M)$ are related by the formula
\begin{equation*}
	\Sigma(M) = \frac{1}{n+1} \chi(M).
\end{equation*}
\end{cor}

\begin{proof}
By Hirzebruch's Signature Theorem (\cite[Theorem 8.2.2]{Hirzebruch-TopAlgGeomBook} or \cite[Theorem 19.4]{MilnorStasheff}), one knows that $\Sigma(M)$ can be written as a linear combination of the Pontrjagin numbers.  
A specific formula is given by the {\it L-genus} of $M$, but that is not important here.  
It is well known \cite[Corollary 15.5]{MilnorStasheff} that each Pontrjagin number can be written as a linear combination of Chern numbers.  
Via substitution, one has that
\begin{equation*}
	\Sigma(M) = \sum_{I \in [n]} \alpha(I, n) c_I(M)
\end{equation*}
where $[n]$ denotes the set of partitions of $n$.  

By assumption, all Chern numbers of $M$ can be written as a specific constant times the Euler characteristic of $M$ as noted above.  
This gives
\begin{equation}\label{eqn:signature and Euler characteristic}
	\Sigma(M) = f(n) \chi(M)
\end{equation}
where the coefficient $f(n)$ is independent of $M$.
Equation \eqref{eqn:signature and Euler characteristic} holds for all manifolds $M$ whose ratio of Chern numbers is equal to the corresponding ratio of Chern numbers for $\C \P^n$.  
In particular, it holds for $\C \P^n$.  
It is again well known that $\Sigma(\C \P^n) = 1$ \cite[Page 225]{MilnorStasheff} and $\chi(\C \P^n) = n+1$ (\cite[Theorem 14.4]{MilnorStasheff}, or is obvious from the standard CW-structure of $\C \P^n$).  
Solving for $f$ gives the desired formula.

\end{proof}

\section{The signature of a ramified covering and proofs of main theorems}\label{sect:proofs}
In this section we review the main results of \cite{Hirzebruch-RamifiedCoverings} and \cite{Viro}, and then prove the statements from the Introduction.
To this end, let us first recall the set-up for Theorems \ref{thm:different chern numbers} and \ref{thm:dimension 2}.  
Let $n \geq 1$ be an even integer and $d \geq 2$ an integer.
Suppose that $(M, N)$ is modeled on $(\C \H^n, \C \H^{n-1})$ with $M$ closed.  
Let $(M', N')$ be a finite cover of degree $m$ such that $[N'] \in H_{n-2}(M', \mathbb{Z})$ is $d$-divisible.
Let $X$ denote the $d$-fold cyclic branched cover of $M'$ about $N'$.  
We use $Y$ to denote the branching locus in $X$, which is isometric to $N'$.
The general situation is as follows.
\begin{equation*}
	(X,Y) \underset{\text{d-fold branched cover}}{\looooongrightarrow} (M', N') \underset{\text{degree m regular cover}}{\looooongrightarrow} (M, N).
\end{equation*}

\subsection{The signature of $X$ in terms of $M'$ and $N'$}
First, almost all of the material from this subsection comes from \cite{Hirzebruch-RamifiedCoverings} and \cite{Viro}.
Where possible, we try to use the same notation as \cite{Hirzebruch-RamifiedCoverings} for readability.

Consider the rational function
\begin{equation*}
	\text{sign}(t) = \frac{(1+t)^d + (1-t)^d}{(1+t)^d - (1-t)^d} \cdot t.
\end{equation*}
In \cite[Section 5]{Hirzebruch-RamifiedCoverings} Hirzebruch proves that this function determines the signature of $X$ in terms of $M'$ and $N'$, where $d$ as above denotes the degree of the branched cover.  
How this is interpreted is as follows.  
Expand sign$(t)$ as a formal power series in $t$.  
The first few terms of this series are
\begin{equation*}
\text{sign}(t) = \frac{1}{d} + \frac{d^2 - 1}{3d} t^2 - \frac{(d^2-1)(d^2-4)}{45d} t^4 + \hdots.
\end{equation*}
Then
\begin{equation}\label{eqn:signature formula}
\Sigma(M') = \frac{1}{d} \Sigma(X) + \frac{d^2 - 1}{3d} \Sigma(Y_2) - \frac{(d^2-1)(d^2-4)}{45d} \Sigma(Y_4) + \hdots .
\end{equation}
where, for each positive integer $r$, $Y_r$ is a submanifold of $X$ whose construction we explain below.
But first note that one can rewrite equation \eqref{eqn:signature formula} as
\begin{equation}\label{eqn:signature formula 2}
\Sigma(X) = d \Sigma(M') - \frac{d^2 - 1}{3} \Sigma(Y_2) + \frac{(d^2-1)(d^2-4)}{45} \Sigma(Y_4) - \hdots .
\end{equation}

The submanifolds $Y_r$ are just transverse intersections of perturbations of the submanifold $Y$, which we define recursively.  
First, $Y_1 = Y$.  
Now, assuming that $Y_r$ is defined for some positive integer $r$, we can perturb $Y_r$ to obtain a submanifold $Y_r'$ which is transverse to $Y$. 
Define $Y_{r+1} = Y_r' \cap Y$.
The signature of such a manifold is independent of the choice of (transverse) perturbation, and note that this procedure can be accomplished in such a way so that $Y_{r+1} \subset Y_r$ for each $r$.

Notice that $Y_{r+1}$ will have codimension $2$ within $Y_r$.  
Therefore, if $r$ is even, then $Y_r$ will have dimension a multiple of $4$.  
Thus, all signatures in \eqref{eqn:signature formula 2} are potentially non-zero.
Also, for $n$ even, the maximal number of nonzero terms on the right-hand side of \eqref{eqn:signature formula 2} is $1 + (n/2)$.

\begin{remark}\label{rmk:comparison of normal bundles}
Let $\nu$ denote the normal bundle of $Y$ in $X$, and let $\nu'$ denote the normal bundle of $N'$ in $M'$.  
Since $Y \cong N'$ we may consider both of these bundles as being over $N'$, where it follows immediately that $\nu' = \nu^d$. 
Both bundles $\nu'$ and $\nu$ can be extended to complex line bundles over $M'$, which we denote $E'$ and $E$, respectively.  
From our above considerations we have that $E' = E^d$, from which follows that $c_1(E') = d c_1(E)$. 
\end{remark}

\begin{remark}\label{rmk:top Chern class} 
We wish to consider the top Chern class $c_r(Y_r^\perp)$, where $Y_r^\perp$ denotes the normal bundle of $Y_r$ in $X$.
Notice that $Y_r^\perp$ splits as a direct sum of the portion that is tangent to $Y$, denoted $Y_r^\perp \bigr|_Y$, and the part that is orthogonal to $Y$, denoted $Y^{\perp}\bigr|_{Y_r}$.  
We write this as
	\begin{equation*}
	Y_r^\perp =  Y_r^\perp \bigr|_Y \oplus Y^{\perp}\bigr|_{Y_r}.
	\end{equation*}
By the Chern class version of the Whitney sum formula \cite[Equation 14.7]{MilnorStasheff} we have, for the top Chern classes, 
	\begin{equation*}
	c_r(Y_r^\perp) = c_{r-1}(Y_r^\perp \bigr|_Y) \cdot c_1 \left( Y^\perp \bigr|_{Y_r} \right).
	\end{equation*}
Since $Y_r \subset Y$ for each $r$, and since $Y \cong N'$, there is an associated submanifold $N_r'$ of $N'$ with $N_r' \cong Y_r$. 
The bundle $Y_r^\perp \bigr|_Y$ is naturally isomorphic to the normal bundle of $N_r'$ within $N'$, denoted $(N_r')^\perp \bigr|_{N'}$.
Also, in a similar manner as to Remark \ref{rmk:comparison of normal bundles}, we have $d \, c_1(Y^\perp \bigr|_{Y_r}) = c_1 (  (N')^\perp \bigr|_{N'_r} ) $.
Thus
	\begin{equation}\label{eqn:Chern class reduction}
	d \, c_r \left( Y_r^\perp \right) = c_{r-1} \left( (N_r')^\perp \bigr|_{N'} \right) \cdot c_1 \left( (N')^\perp \bigr|_{N_r'} \right) = c_r \left((N_r')^\perp \right).
	\end{equation}
Exact values for the Chern classes $c_r((N_r')^\perp )$ seem difficult to calculate, but note that by \cite[Lemma 13.1]{BelegradekCH} we have $c_1((N')^\perp) \neq 0$.  
Accordingly, one would expect that each of these Chern classes are typically nonzero in appropriate dimensions.
\end{remark}

\begin{remark}
There is a small error in the general formula in \cite[Equation 14]{Hirzebruch-RamifiedCoverings}.
The error is minor, and the author should have caught it sooner.  
But it should be noted that this is fixed by Viro in \cite[Equation 3]{Viro}.
This corrected form was, of course, used to derive equation \eqref{eqn:signature formula 2} above.
\end{remark}

When $n=2$, notice that equation \eqref{eqn:signature formula 2} reduces to
\begin{equation}\label{eqn:reduction to dimension 2}
\Sigma(X) = d \Sigma(M') - \frac{d^2 - 1}{3d} e((N')^\perp)
\end{equation}
where $e((N')^\perp)$ denotes the Euler number of the normal bundle of $N'$ within $M'$.   
Since $n=2$ and $N'$ is a totally geodesic complex hypersurface in $M'$, by \cite[Proposition 2.5]{GoldmanKapovichLeeb} we have that $e((N')^\perp) = (1/2) \chi(N')$.
Thus
\begin{equation}\label{eqn:n=2 case}
\Sigma(X) = d \Sigma(M') - \frac{d^2-1}{6d} \chi(N').
\end{equation}

\subsection{Proofs of the main theorems}

We first consider Theorem \ref{thm:dimension 2}.  
Note that, when $n=2$, the only two Chern numbers are $c_2$ and $c_1^2$.  
For a closed complex hyperbolic manifold, we have
\begin{equation*}
	\frac{c_1^2}{c_2} = \frac{ {3 \choose 1}^2}{{3 \choose 2}} = 3 \quad \Longrightarrow \quad c_1^2 - 3c_2 = 0.
\end{equation*}
So Theorem \ref{thm:dimension 2} shows that the only ratio of Chern numbers of $X$ is not equal to the corresponding ratio for a complex hyperbolic manifold.

\begin{proof}[Proof of Theorem \ref{thm:dimension 2}]
The formula for the first Pontrjagin number $p_1(X)$ of any complex vector bundle in terms of Chern numbers is $p_1(X) = c_1^2(X) - 2c_2(X)$ \cite[Corollary 15.5]{MilnorStasheff}.  
Also, for any $4$-manifold $X$, $\Sigma(X) = (1/3) p_1(X)$ \cite[page 225]{MilnorStasheff}.
So we have
\begin{align}
c_1^2(X) - 3c_2(X) &= p_1(X) - c_2(X) \notag \\
&= 3 \Sigma(X) - \chi(X) \notag \\
&= 3 \Sigma(X) - \left[ d \chi(M') - (d-1) \chi(N') \right]. \label{eqn:first Chern number equation}
\end{align}
where the last equality follows from a basic inclusion-exclusion argument.

By equation \eqref{eqn:n=2 case} above we also have
\begin{equation}\label{eqn:signature branched cover equation}
\Sigma(X) = d \Sigma(M') - \frac{d^2 - 1}{6d} \chi(N')
\end{equation}
Combining equations \eqref{eqn:first Chern number equation}, \eqref{eqn:signature branched cover equation}, and Corollary \ref{cor:signature Euler characteristic} provides
\begin{align}
c_1^2(X) - 3c_2(X) &= 3 \left[ d \Sigma(M') - \frac{d^2 - 1}{6d} \chi(N') \right] - d \chi(M') + (d-1) \chi(N') \notag \\
&= 3 \left[ \frac{d}{3} \chi(M') - \frac{d^2 - 1}{6d} \chi(N') \right] - d \chi(M') + (d-1) \chi(N') \notag \\
&= \frac{(d-1)^2}{2d} \chi(N') \notag \\
&= m \frac{(d-1)^2}{2d} \chi(N). \label{eqn:final calculation when n=2}
\end{align}
By Chern-Gauss-Bonnet we know that $\chi(N) \neq 0$.  
So $c_1^2(X) - 3c_2(X) \neq 0$ for all $d \geq 2$, and the magnitude of this difference increases without bound as $m \to \infty$.
\end{proof}

\begin{remark}
Note that equation \eqref{eqn:final calculation when n=2} proves Corollaries \ref{cor:corollary 1} and \ref{cor:corollary 2} when $n=2$. 
\end{remark}

\begin{proof}[Proof of Theorem \ref{thm:different chern numbers}]
We proceed by contradiction.  
Assume that all ratios of Chern numbers of $X$ are equal to the corresponding ratio of Chern numbers of $\C \P^n$.
By Corollary \ref{cor:signature Euler characteristic} we then have that 
\begin{equation*}
\Sigma(X) = \frac{1}{n+1} \chi(X) = \frac{1}{n+1} \left[ d \chi(M') - (d-1) \chi(N') \right].
\end{equation*}
On the other hand, by equation \eqref{eqn:signature formula 2} and Corollary \ref{cor:signature Euler characteristic} we have
\begin{align*}
\Sigma(X) &= d \Sigma(M') - \frac{d^2 - 1}{3} \Sigma(Y_2) + \frac{(d^2-1)(d^2-4)}{45} \Sigma(Y_4) - \hdots  \\
&= \frac{d}{n+1} \chi(M') - \frac{d^2 - 1}{3} \Sigma(Y_2) + \frac{(d^2-1)(d^2-4)}{45} \Sigma(Y_4) - \hdots .
\end{align*}
Setting these equations for $\Sigma(X)$ equal yields
\begin{equation}\label{eqn:equal 0 equation 1}
0 = \frac{d-1}{n+1} \chi(N') - \frac{d^2 - 1}{3} \Sigma(Y_2) + \frac{(d^2-1)(d^2-4)}{45} \Sigma(Y_4) - \hdots .
\end{equation}

Now, consider the submanifold $Y_{2r}$.  
One could have alternatively defined this submanifold to be $Y_r \cap Y_r'$, where $Y_r'$ is a copy of $Y_r$ that was perturbed to be transverse to $Y_r$.
The signatures of these two definitions agree.  
But, by using this definition for $Y_{2r}$, one sees by \cite[Corollary 6.3]{Fulton} that $\Sigma(Y_{2r})$ is equal to the top Chern number of the normal bundle to $Y_r$.  
Equation \eqref{eqn:equal 0 equation 1} then becomes
\begin{equation*}
0 = \frac{d-1}{n+1} \chi(N') - \frac{d^2 - 1}{3} c_1(Y_1^\perp) + \frac{(d^2-1)(d^2-4)}{45} c_2(Y_2^\perp) - \hdots .
\end{equation*}
By equation \eqref{eqn:Chern class reduction}, this can be rewritten further as
\begin{align*}
0 &= \frac{d-1}{n+1} \chi(N') - \frac{d^2 - 1}{3d} c_1((N_1')^\perp) + \frac{(d^2-1)(d^2-4)}{45d} c_2((N_2')^\perp) - \hdots .
\end{align*}
Because there are at most $1+(n/2)$ nonzero summands on the right-hand side of this equation, and since the Euler characteristic and all Chern numbers are independent of $d$, it is clear that the above equation is only true for at most finitely many values of $d$.
To help prove the corollaries just note that one can rewrite this expression further as
\begin{align}
0 &= m \frac{d-1}{n+1} \chi(N) - \frac{d^2 - 1}{3d} f_1(m) c_1(N_1^\perp) \label{eqn:proof of Cor 1}  \\
& + \frac{(d^2-1)(d^2-4)}{45d} f_2(m) c_2(N_2^\perp) - \hdots \notag
\end{align}
for some functions $f_i$ of $m$, and where $N_r$ is defined analogously to $N_r'$.  

\end{proof}

\begin{remark}
Note that, if one could compute explicit values for the Chern classes in \eqref{eqn:proof of Cor 1}, then one could likely prove that this expression is never $0$ for all $d \geq 2$.
\end{remark}

\begin{proof}[Proof of Corollary \ref{cor:corollary 1}]
By \cite{MinemyerKahler} one can construct a sequence of finite covers $(M_k, N_k)$ of $(M,N)$, with $d$-fold cyclic branched covers $X_k$, which admit a Riemannian metric $g_k$ whose sectional curvatures all lie in the interval $[-1-(1/k), (-1/4) + (1/k)]$.
Let $m_k$ denote the degree of each finite cover.  
Note that the right-hand side of \eqref{eqn:proof of Cor 1}, with $m_k$ substituted for $m$, approaches $\pm \infty$ as $k \to \infty$.  
If the ratio of all Chern numbers of $X_k$ approached those of $\C \P^n$ as $k \to \infty$, then the right-hand side of \eqref{eqn:proof of Cor 1} would approach $0$ as $k \to \infty$.  
\end{proof}

Corollary \ref{cor:corollary 2} follows immediately from Corollary \ref{cor:corollary 1}.



\vskip 20pt


\bibliographystyle{plain}
\bibliography{references.bib} 


\end{document}